\documentstyle[amsfonts,amssymb,amsthm,amsmath, color, 12pt]{article}

\title{On n-Hausdorff homogeneous and n-Urysohn homogeneous spaces}
\author{M. Bonanzinga\footnote{MIFT Department, University of Messina, Italy, mbonanzinga@unime.it .}, N. Carlson\footnote{Department of Mathematics, California Lutheran University, USA,  ncarlson@callutheran.edu .}, D. Giacopello\footnote{MIFT Department, University of Messina, Italy, dagiacopello@unime.it .}, F. Maesano\footnote{MIFT Department, University of Messina, Italy, fomaesano@unime.it .}}
\date{}

\newtheorem{theorem}{Theorem}[section]

\newtheorem{question}[theorem]{Question}
\newtheorem{example}[theorem]{Example}
\newtheorem{lemma}[theorem]{Lemma}
\newtheorem{prop}[theorem]{Proposition}

\newtheorem{definition}[theorem]{Definition}
\newtheorem{remark}[theorem]{Remark}

\begin{document}
\maketitle
\begin{abstract}
In this paper we study $n$-Hausdorff homogeneous and $n$-Urysohn homogeneous spaces. We give some upper bounds for the cardinality of these kind of spaces and give examples. Additionally we show that for every $n>2$, there is no $n$-Hausdorff 2-homogeneous space. Finally, for any $n$-Hausdorff space we construct an $n$-Hausdorff homogeneous extension which is the union of countably many $n$-H-closed spaces. 
\end{abstract}



\section{Introduction}

Throughout the paper, $n$ will always denote an integer. Given a topological space $X$, 
the \emph{Hausdorff number} $H(X)$ (finite or infinite) of $X$ is the least cardinal number $\kappa$ such that for every subset $A\subseteq X$ with $|A| \geq \kappa$ there exist open neighbourhoods $U_{a}$, $a \in A$, such that $\bigcap_{a \in A} U_{a} = \emptyset$. A space $X$ is  said \emph{n-Hausdorff}, $n \geq 2$, if $H(X) \leq n$. Of course,  with $|X|\geq 2$, $X$ is Hausdorff iff $n=2$ \cite{B}; 
the \emph{Urysohn number} $U(X)$ (finite or infinite) of $X$ is the least cardinal number $\kappa$ such that for every subset $A\subseteq X$ with $|A| \geq \kappa$ there exist open neighbourhoods $U_{a}$, $a \in A$, such that $\bigcap_{a \in A} \overline{U_{a}} = \emptyset$. A space $X$ is  said \emph{n-Urysohn}, $n \geq 2$, if $U(X) \leq n$. Of course,  with $|X|\geq 2$, $X$ is Urysohn iff $n=2$ (see \cite{BCM1,BCM2}).\\ 

A space $X$ is \textit{homogeneous} if for every $x, y \in X$ there exists a homeomorphism $h : X \rightarrow X$ such that $h(x) = y$  (see \cite{AM, CSurvey} for surveys on homogeneous spaces). 

\begin{definition} \rm \cite{CPR2}
	A space $X$ is $2$-homogeneous if for every $x_1,x_2,y_1,y_2\in X$ there exists a homeomorphism $h : X \rightarrow X$ such that $h(x_1) = y_1$ and $h(x_2) = y_2$.
\end{definition}
In general one can give the definition of $n$-homogeneous space for any $n$. Notice that 1-homogeneity coincides with the definition of homogeneity. Of course, if a space is $(n+1)$-homogeneous, then it is $n$-homogeneous for every $n\geq 1$.\\

In this paper we prove that $n$-Hausdorff ($n>2$) non Hausdorff spaces are not $m$-homogeneous ($m>1$) and give an example (Example \ref{bing}) of a 3-Urysohn homogeneous non Urysohn space. Also we show that  even in the class of homogeneous spaces $(n+1)$-Hausdorff  ($(n+1)$-Urysohn) spaces need not be $n$-Hausdorff (resp., $n$-Urysohn), with $n \geq 2$. Also we present some upper bounds on the cardinality of $n$-Hausdorff homogeneous and $n$-Urysohn homogeneous spaces (see also \cite{B,BCG} for other bounds on the cardinality of n-Hausdorff and n-Urysohn spaces). In particular, we prove the analogous version of the following result for $n$-Urysohn spaces and a variation of the same result for $n$-Hausdorff spaces.

\begin{theorem}\rm\cite{CR} \label{CR}
Let $X$ be a homogeneous Hausdorff space. Then $|X| \leq 2^{c(X) \pi \chi (X)}$.
\end{theorem}

In the last section of the paper, for any $n \geq 2$ and for any $n$-Hausdorff space, we construct an $n$-Hausdorff homogenous extension which is the countable union of $n$-H-closed spaces. Using this result we give an example of $n$-Hausdorff homogeneous space which is not $n$-Urysohn, for every $n \geq 2$.

We consider cardinal invariants of topological spaces (see \cite{E, H, J}) and all
cardinal functions are multiplied by $\omega$. In particular, given a topological space $X$, we will denote with $d(X)$ its density, $\chi(X)$ its character, $\pi \chi(X)$ its $\pi$-character, $\pi w(X)$ its $\pi$-weight and $c(X)$ its cellularity. Recall also that, for any space $X$, $d(X) \pi \chi (X) = \pi w (X)$.
\smallskip

Recall that a family $\mathcal{U}$ of open sets of a space $X$ is \emph{point-finite} if for every $x \in X$, the set $\{ U \in  {\cal U} : x \in U\}$ is finite \cite{E}. Tkachuck \cite{T} defined $p(X) = sup\{|\mathcal{U}|: \mathcal{U}$ is a point-finite family in $X \}$. 
In \cite{B}, Bonanzinga introduced the following definition:

\begin{definition}\rm\cite{B}
A family $\mathcal{U}$ of open sets of a space $X$ is \emph{point-($\leq n$) finite}, where $n \in \mathbb{N}$, if for every $x \in X$, the set $\{ U \in  {\cal U} : x \in U\}$ has cardinality $\leq n$. For each $n \in \mathbb{N}$, put 
\begin{center}
$p_{n}(X) = sup\{ |\mathcal{U}|: \mathcal{U}$ is a point-($\leq n$) finite family in $X \}.$
\end{center}
\end{definition}

\begin{prop}\rm\cite{B}
\label{P1}
Let $X$ be a topological space. Then $p_{n}(X)=c(X)$ for every $n \in \mathbb{N}$.
\end{prop}

\section{Examples and positive results}

In \cite{B},  examples of $(n + 1)$-Hausdorff spaces which are not $n$-Hausdorff, for
every $n \geq 2$, and an example of a space $X$ such that $H(X) = \omega$ and $H(X)\neq n$, for each $n \geq 2$, are given. Also, in \cite{BCM1} examples of Hausdorff  $(n + 1)$-Urysohn spaces which are not $n$-Urysohn were given for every $n\geq 2$.\\

Recall that a hyperconnected (or nowhere Hausdorff) space is a space such that the intersection of two nonempty open sets is nonempty; a space is nowhere Urysohn if there is no pair of nonempty open sets with disjoint closures.\\
\smallskip

\begin{prop}\rm\label{yes hyper}
	A non Hausdorff $2$-homogeneous space is hyperconnected.
\end{prop}
\begin{proof}
	Let $X$ be a non Hausdorff $2$-homogeneous space. Suppose that there are two nonempty open subset $V_1$ and $V_2$ of $X$ such that $V_1\cap V_2=\emptyset$. Fix two points $y_1\in V_1$ and $y_2\in V_2$. Since $X$ is not Hausdorff there exist two points $x_1,x_2\in X$ such that for every open neighbourhood $U_1$ of $x_1$ and $U_2$ of $x_2$, one has that $U_1\cap U_2\not=\emptyset$. Define the homeomorphism $h:X\to X$ such that $h(x_1)=y_1$ and $h(x_2)=y_2$. Of course $h^{\leftarrow}(V_1)\cap h^{\leftarrow}(V_2)\not=\emptyset$. Pick a point $x\in h^{\leftarrow}(V_1)\cap h^{\leftarrow}(V_2)$, then $h(x)\in V_1\cap V_2$, a contradiction.
\end{proof}

\begin{prop}\rm
	A non Urysohn $2$-homogeneous space is nowhere Urysohn.
\end{prop}
\begin{proof}
	The proof is similar to the one of Proposition \ref{yes hyper}. One just needs to consider that if $h:X\to X$ is a homeomorphism, then $h(\overline{A})=\overline{h(A)}$ for each $A\subseteq X$.
\end{proof}

The following proposition follows directly from the definition.
 
\begin{prop}\rm\label{hyper}
	A space $X$ is hyperconnected if and only if for every finite $A \subseteq X$, $|A|=n$, $n \geq 2$, and for every choice of neighbourhoods $U_{a}$, $a \in A$, $\bigcap_{a \in A} U_{a} \neq \emptyset$.
\end{prop}

By Proposition \ref{hyper} one can easily show the following.

\begin{prop}\rm\label{no Hyper}
	Let $n\geq 2$. Any $n$-Hausdorff space is not hyperconnected.
\end{prop}

\begin{theorem}\rm\label{non exists}
	There is no $n$-Hausdorff non Hausdorff $m$-homogeneous space for every $n>2$ and every $m>1$.
\end{theorem}
\begin{proof}
	It follows directly from propositions \ref{no Hyper} and \ref{yes hyper}.
\end{proof}

The following example shows that there exist 3-Hausdorff homogeneous spaces.

\begin{example}\label{countEx}\rm
	A countable $3$-Hausdorff homogeneous space. 
\end{example}
Consider the space $X$ of natural numbers with the topology generated the base $\{\{n,n+1\} : n\hbox{ is even}\}$. $X$ is a $3$-Hausdorff homogeneous space. \hfill$\triangle$\\

Note that the space of the previous example is a homogeneous space which is not $2$-homogeneous.\\

The analougues Proposition \ref{no Hyper} and Theorem \ref{non exists} for $n$-Urysohn spaces do not hold, as the following example shows.

\begin{example}\rm\label{bing}
	A $3$-Urysohn homogeneous space which is not Urysohn.
\end{example}
Consider the well known \lq\lq irrational slope space", also called Bing's Tripod space (see \cite[Example 75]{SS}). This space is $n$-homogeneous, $n \geq 1$  \cite{BBHS}, and $3$-Urysohn.  \hfill$\triangle$\\

Recall that for every $n \geq 2$ there exist examples of $(n+1)$-Hausdorff spaces which are not $n$-Hausdorff \cite{B}, and examples of $(n+1)$-Urysohn spaces which are not $n$-Urysohn \cite{BCM2}. Then it is natural to pose the following questions.

\begin{question}\rm\label{q2}
	Is every $(n+1)$-Hausdorff homogeneous space $n$-Hausdorff, for each $n\geq 2$?
\end{question}

\begin{question}\rm\label{q5}
	Is every $(n+1)$-Urysohn homogeneous space $n$-Urysohn, for each $n \geq 2$?
\end{question}

Examples \ref{countEx} and \ref{bing} answer negatively to questions \ref{q2} and \ref{q5} respectively for $n = 2$. Note that the space of Example \ref{countEx} is $3$-Urysohn, and the construction can be generalized to obtain $(n+1)$-Urysohn non $n$-Hausdorff spaces for each $n \geq 2$.\\

In \cite{B}, Bonanzinga gives an example of an $\omega$-Hausdorff space which is not $n$-Hausdorff for every $n \geq 2$. Now we give a countable $\omega$-Hausdorff homogeneous space which is not $n$-Hausdorff for every $n\geq 2$.

\begin{example}\rm
There is a countable $T_{1}$ hyperconnected (hence not $n$-Hausdorff for every $n \geq 2$) space, which is $\omega$-Hausdorff and homogeneous.
\end{example}
In \cite{BSS}, the following space is constructed. Let $X = \mathbb{Z} \times \mathbb{Z}$ and $\mathcal{B}$ $= \{ U_{j,k}, V_{j,k} : j,k \in \mathbb{Z}\}$ is the subbase for the topology, where
\begin{itemize}
\item[]$U_{j,k}=\{(x,y) \in {\Bbb Z}^{2} : x > j$ or $y > k \}$
\item[]$V_{j,k}=\{(x,y) \in {\Bbb Z}^{2} : x < j$ or $y < k \}.$
\end{itemize}
This is a $T_{1}$ hyperconnected, hence not $n$-Hausdorff space for every $n \geq 2$ which is $\omega$-Hausdorff, homogeneous, first countable, Lindelof.\hfill $\triangle$\\

In \cite{BCM2}, Bonanzinga, Cammaroto and Matveev constructed an Hausdorff space with extent equal to $\kappa$, $\kappa \geq \omega$, which is not $\kappa$-Urysohn (we give this example for sake of completeness, see Example \ref{HardEx} below). The construction of such a space may be considered a modification of irrational slope space \cite[Example 75]{SS}. Since the irrational slope space is homogeneous, it is natural to ask the following.

\begin{question}\rm Is the space in Example \ref{HardEx} homogeneous?
\end{question}

\begin{example}\rm
\label{HardEx}
For every cardinal $\kappa$ there exists a Hausdorff space with extent equal to $\kappa$, $\kappa \geq \omega$, which is not $\kappa$-Urysohn.
\end{example}
Let $˜\tilde{D} = \{d_{\alpha, n} : \alpha < \kappa ,n \in \omega \}$ be a discrete space of cardinality $\kappa$, and $D = \tilde{˜D} \cup \{ p \}$ the one point compactification of $\tilde{˜D}$. Put $E = D \cup \{ d^{*} \}$ where $d^{*}$ is isolated in $E$ and is not in $D$. Consider $\kappa^{+}$ with order topology, $D \times \kappa^{+}$ with the Tychonoff product topology, and denote $W = \{p\} \times \kappa^{+}$; then $W$ is a subspace of $D \times \kappa^{+}$ homeomorphic to $\kappa^{+}$. Also, for $\alpha < \kappa^{+}$ denote $W_{\alpha} = \{p\} \times [\alpha, \kappa^{+})$. For $\alpha < \kappa$, $\beta < \kappa^{+}$, denote $D_{\alpha} = \{d_{\alpha,n} : n \in \omega\}$, and $T_{\alpha,\beta} = D_{\alpha} \times [\beta, \kappa^{+}) \subset D \times \kappa^{+}$. Let $\vec{p}$ be the point in $E^{\kappa^{+}}$ with all coordinates equal to $p$. Let $S = \{x \in E^{\kappa^{+}} : |\{\alpha < \kappa^{+} : x(\alpha) \neq p\}| \leq \kappa\}$ be the $\Sigma_{k}$-product in $E^{\kappa^{+}}$ with center at $\vec{p}$. It can be proved that there is a homeomorphic embedding $f : D \times \kappa^{+} \rightarrow E^{\kappa^{+}}$ such that
\begin{itemize}
\item[(1)] $f(D \times \kappa^{+}) \cap S = f(W)$.
\item[(2)] $f(W)$ is closed in $S$ and homeomorphic to $\kappa^{+}$ with the order topology.
\item[(3)] for every distinct $\alpha, \gamma < \kappa$, the sets $f(T_{\alpha,0})$ and $f(T_{\gamma,0})$ can be separated by open neighbourhoods in $E^{\kappa^{+}}$.
\item[(4)] $\overline{f(T_{\alpha,\beta})} \cap S = f(W_{\beta})$.
\end{itemize}
Finally, let $L = \{l_{\alpha} : \alpha < \kappa \}$ (where all points $l_{\alpha}$ are distinct) be a set disjoint from $E^{\kappa^{+}}$ and topologize $X = S \cup L$ as follows: $S$, with the topology inherited from $E^{\kappa^{+}}$ is open in $X$; a basic neighbourhood of $l_{\alpha}$ takes the form $\{l_{\alpha}\} \cup (U \cap S)$ where $U$ is arbitrary neighbourhood (in $E^{\kappa^{+}}$) of $f(T_{\alpha,\beta})$ for some $\beta < \kappa^{+}$. We recall that $L$ is closed discrete in this space, so $e(X) \geq \kappa$, and for every family $\{U_{l} : l \in L\}$ of neighbourhoods of points $l \in L$ in $X$, $\bigcap \{\overline{U_{l}} : l \in L\}  \neq \emptyset$, so it is not $\kappa$-Urysohn.\hfill $\triangle$\\

\section{On the cardinality of n-Hausdorff homogeneous and $n$-Urysohn homogeneous spaces.}
In \cite{HJ}, Hajnal and Juha\'sz proved that, for every Hausdorff space $X$,  $|X| \leq 2^{c(X)\chi(X)}$. In \cite{B} Bonanzinga proved that $|X|\leq 2^{2^{c(X)\chi(X)}}$ for every $3$-Hausdorff space $X$ and asked if $|X|\leq 2^{c(X)\chi(X)}$ holds for every $n$-Hausdorff space $X$, with $n\geq 2$. In \cite{G} Gotchev, using the cardinal function called \lq \lq non Hausdorff number" introduced independently from \cite{B}, gave a positive answer to the previous question.\\

In \cite{CR}, Carlson and Ridderbos proved the following result.

\begin{theorem}\rm\cite{CR} \label{CRth}
Let $X$ be a homogeneous Hausdorff space. Then $|X| \leq 2^{c(X) \pi \chi (X)}$.
\end{theorem}

In fact, in \cite{CR} it is proved that the previous theorem holds for power homogeneous Hausdorff spaces. Recall that a topological space $X$ is power homogeneous if $X^\mu$ is homogeneous for some cardinal number $\mu$.  Clearly, if a space is homogeneous it is power homogeneous.\\

Then, it is natural to pose the following question.

\begin{question}\label{Q36}\rm
	 Is $|X|\leq 2^{c(X)\pi\chi(X)}$ true for every homogeneous space $X$ such that $H(X)$ is finite?
\end{question}

In the following we give partial answers to the previous question.\\

Given a set $A$ and a cardinal $\kappa$, $[X]^{\kappa}$ denotes the set of all subsets of $A$ whose cardinality is $\kappa$.\\

\begin{theorem}\rm \label{ER} \cite{ER}
	 Let $\kappa$ be a cardinal number and $f: [(2^{2^{\kappa}})^{+}]^3 \to \kappa$ a function, then there exists a subset $H\in  [(2^{2^{\kappa}})^{+}]^{\kappa^+} $ such that $f\restriction [H]^3$ is constant.
	\end{theorem}

\begin{theorem}\rm
Let $X$ be a $3$-Hausdorff homogeneous space. Then
\begin{center}
$|X| \leq 2^{2^{c(X) \pi \chi (X)}}$ 
\end{center}
\begin{proof}
Let $c(X) \pi \chi (X) = \kappa$. Then, by Proposition \ref{P1}, we have $p_{2}(X) \leq \kappa$. Suppose that $|X| \geq (2^{2^{\kappa}})^{+}$. For every triple $x_{1}, x_{2}, x_{3} \in X$ of distinct points  select neighbouroods $U_{i}(x_{1}, x_{2}, x_{3})$ of $x_{i}$ for $i=1,2,3$ such that\\
$\bigcap_{i=1}^{3} U_{i}(x_{1}, x_{2}, x_{3}) = \emptyset$. Fix a point $p \in X$ and a local $\pi$-base $\mathcal{B}$ for $p$ with $|\mathcal{B}|=\kappa$. Since the space is homogeneous, there exists a family $\{ h_{x} \}_{x \in X}$ of homeomorphisms $h_{x} : X \rightarrow X$ such that $h_{x}(p) = x$ for every $x \in X$. Fix distinct points $x_{1}, x_{2}, x_{3} \in X$ and observe that the set $\bigcap_{i=1}^{3} h^{\leftarrow}_{x_{i}}(U_{i}(x_{1}, x_{2}, x_{3}))$ is an open neighbourhood of $p$; since $\mathcal{B}$ is a $\pi$-base, there is a non empty $B(x_{1}, x_{2}, x_{3}) \in \mathcal{B}$ such that $B(x_{1}, x_{2}, x_{3})$ is contained in it. Consider now the function $f : [X]^{3} \rightarrow \mathcal{B}$ defined by $f(\{x_{1}, x_{2}, x_{3} \}) = B(x_{1}, x_{2}, x_{3})$. Then by Theorem \ref{ER} there is $Z \in [X]^{\kappa^{+}}$ and $B \in \mathcal{B}$ such that $f\restriction [Z]^{3}= \{B \}$. 

Now, the family $\{ h_{z}(B) : z \in Z \}$ is point-($\leq 2$) finite in $X$. To see this, suppose by way of contradiction that there exists $x_{0} \in X$ such that $|\{ h_{z}(B) : x_{0} \in h_{z}(B)\}|=3$. So there are $z_{1}, z_{2}, z_{3} \in X$ such that $x_{0} \in h_{z_{i}}(B),\hspace{1mm}  i=1,2,3$. This implies $x_{0} \in h_{z_{i}}(B) \subseteq h_{z_{i}}(\bigcap\limits_{i=1}^{3} h_{z_{i}}^{\leftarrow}(U_{i}(z_{1}, z_{2}, z_{3}))) \subseteq h_{z_{i}}( h_{z_{i}}^{\leftarrow}(U_{i}(z_{1}, z_{2}, z_{3})))= U_{i}(z_{1}, z_{2}, z_{3})$. Then, $x_{0} \in \bigcap\limits_{i=1}^{3} U_{i}(z_{1}, z_{2}, z_{3}) \neq \emptyset$, a contradiction.

Furthermore, $\{ h_{z}(B) : z \in Z \}$ has cardinality exactly $\kappa^{+}$. Otherwise there exists $z_{0} \in Z$ s.t. $|\{ z \in Z : h_{z}(B) = h_{z_{0}}(B) \}|=\kappa^{+}$. As before, from $h_{z_{0}}(B)\subseteq U_{i}(z_{1}, z_{2}, z_{3})$ for every triple of elements in $\{ z \in Z : h_{z}(B) = h_{z_{0}}(B) \}$  we obtain a contradiction.

Thus $p_2(X)=\kappa^+$, a contradiction with $p_{2}(X) \leq \kappa$. This concludes the proof.
\end{proof}
\end{theorem}

Recall the following result.

\begin{theorem}\label{ER2}\rm
	Let $\kappa$ be a cardinal number, $n\geq3$ and  $f: [(2^{2^{.^{.^{.^{2^{\kappa}}}}}})^{+}]^n \to \kappa$ a function (where the power is made $(n-1)$-many times), then there exists a subset $H\in  [(2^{2^{.^{.^{.^{2^{\kappa}}}}}})^{+}]^{\kappa^+} $ such that $f\restriction [H]^n$ is constant.
\end{theorem}

\begin{theorem}\rm\label{n-1power}
Let $X$ be an $n$-Hausdorff homogeneous space, with $n \geq 2$. Then
\begin{center}
$|X| \leq 2^{2^{.^{.^{.^{2^{c(X) \pi \chi (X)}}}}}}$
\end{center}
where the power is made $(n-1)$-many times.
\end{theorem}
\begin{proof}
	Similar to the proof of the previous theorem using Theorem \ref{ER2} instead of Theorem \ref{ER}. 
\end{proof}

Next Theorem \ref{THU} shows that Question \ref{Q36} has a positive answer if $H(X)$ is replaced by $U(X)$.\\

In \cite{CPR}, Carlson, Porter and Ridderbos the following result.

\begin{theorem}\rm\cite{CPR}\label{nHdensity}
	If $X$ is an $n$-Hausdorff homogeneous space, with $n \geq 2$, then $|X|\leq d(X)^{\pi\chi(X)}$. 
\end{theorem}

Also recall that a space is quasiregular if every nonempty open set contains a nonempty regular closed set.

\begin{theorem}\rm If $X$ is an $n$-Hausdorff quasiregular homogeneous space with $n \geq 2$, then $|X|\leq 2^{c(X)\pi\chi(X)}$. 
\end{theorem}

\begin{definition} \rm \cite{V}
Let $X$ be a space. For $A \subseteq X$ , the \emph{$\theta$-closure} of $A$ is defined by
\begin{center}
$cl_{\theta} (A) = \{ x \in  X: \overline{V} \cap  A \neq \emptyset \hbox{ for every open set } V \hbox{containing } x \}$
\end{center}
A set $A \subseteq X$ is \emph{$\theta$-dense} if $cl_{\theta}(A) = X$ . The \emph{$\theta$-density} $d_{\theta}(X)$ of $X$ is defined as the least cardinality of a $\theta$-dense subset of $X$.
\end{definition}

\begin{theorem} \rm \cite{CPR}
\label{Th53}
Let $X$ be an $n$-Urysohn homogeneous space, where $n \geq 2$. Then $|X| \leq d_{\theta}(X)^{\pi \chi (X)}$.
\end{theorem}

\begin{theorem} \rm \cite{C} \label{Carlson}
Let $X$ be a space. Then $d_{\theta}(X) \leq \pi\chi(X)^{c(X)}$.
\end{theorem}

By  Theorems \ref{Th53} and \ref{Carlson}, we obtain the following result.

\begin{theorem}\label{THU}\rm
Let $X$ be an $n$-Urysohn homogeneous space, where $n \geq 2$. Then $|X| \leq 2^{c(X) \pi \chi (X)}$.

\end{theorem}

\begin{proof}
As $X$ is $n$-Hausdorff and homogeneous, we have $|X|\leq d(X)^{\pi\chi(X)}$ by Theorem~\ref{nHdensity}. As $X$ is quasiregular, it follows that $d_\theta(X)=d(X)$. By Theorem~\ref{Carlson},  we have $d(X)\leq\pi\chi(X)^{c(X)}$. Thus, $|X|\leq d(X)^{\pi\chi(X)}\leq\left(\pi\chi(X)^{c(X)}\right)^{\pi\chi{X}}=2^{c(X)\pi\chi(X)}$.
\end{proof}

\section{A homogenous extension of an $n$-Hausdorff space}
In \cite{CPR2} Carlson, Porter and Ridderbos proved the following result.
\begin{theorem}\rm\cite{CPR2}
	Let $X$ be a Hausdorff space. Then $X$ can be embedded
	in a homogeneous space that is the countable union of $H$-closed spaces.
\end{theorem}
In the following we construct (Theorem \ref{Theorem} below) a homogeneous extension of an $n$-Hausdorff space, $n \geq 2$, which is a countable union of $n$-H-closed spaces.

\begin{definition} \rm \cite{BBC}
Let $n \geq 2$. An $n$-Hausdorff space $X$ is called \emph{n-H-closed} if $X$ is closed in every $n$-Hausdorff space $Y$ in which $X$ is embedded.
\end{definition}

Given a space $X$ and an ultrafilter $\mathcal{U}$ on it, we put $a{\cal U}= \bigcap \{ \overline{U} : U \in \mathcal{U} \}$.

For an $n$-Hausdorff space $X$, with $n \geq 2$, an open ultrafilter $\mathcal{U}$ on $X$ is said to be \emph{full} if $|a \mathcal{U}|$ $= n -1$.

\begin{theorem} \rm \cite{BBC} 
Let $n \geq 2$, and $X$ be a space. The following are equivalent:
\begin{itemize}
\item[(a)] $X$ is $n$-Hausdorff
\item[(b)] if $\mathcal{U}$ is an open ultrafilter of $X$, then $|a \mathcal{U}|$ $\leq n -1$.
\end{itemize}

\end{theorem}\begin{theorem} \rm \cite{BBC} 
Let $n \geq 2$, and $X$ be an $n$-Hausdorff space. The following are equivalent:
\begin{itemize}
\item[(a)] $X$ is $n$-H-closed
\item[(b)] every open ultrafilter on $X$ is full.
\end{itemize}
\end{theorem}

Recall the following construction, made in \cite{BBC}. Let $n \geq 2$, $X$ be an $n$-Hausdorff space and ${\frak U} = \{ {\cal U} : {\cal U} \hbox{ is an open ultrafilter such that } |a {\cal U}| < n-1 \}$. We index ${\frak U}$ by ${\frak U} = \{ {\cal U}_{\alpha} : \alpha \in |{\frak U}|\}$. For each $\alpha \in |{\frak U}|$, let $k \alpha = (n-1)-|a{\cal U}_{\alpha}|$ and $\{ p_{\alpha i} : 1 \leq i \leq k \alpha \}$
be a set of distinct points disjoint from $X$. Let $Y = X \cup \{ p_{\alpha i} : 1 \leq i \leq k \alpha, \alpha \in |{\frak U}|\}$. A set $V$ is defined to be open in $Y$ if $V \cap X$ is open in $X$ and if $p_{\alpha i} \in V$ for $1 \leq i \leq k \alpha$, $V \cap X \in {\cal U}_{\alpha}$. The space $Y$ is an $n$-Hausdorff space.

In the following results we use the notation of the previous contruction.
\begin{prop} \rm \cite{BBC}
For every $\alpha\in |{\frak U}|$, 
$${\cal U}_{\alpha} = \{ V \cap X :  p_{\alpha i} \in V\in \tau(Y)\hbox{ for some } 1 \leq i \leq k \alpha\},$$ where $\tau(Y)$ is the topology on $Y$. 
\end{prop}
By the previous proposition the space $Y$ has the property that every open ultrafilter on $Y$ is full. Indeed the points $p_{\alpha i}$, $1 \leq i \leq k \alpha$, added to the space $X$, are in the closure of each element of ${\cal U}_\alpha$. Therefore the space $Y$ is $n$-H-closed. 

\begin{definition}\label{projlarg}  \rm \cite{BBC}
Let $n \geq 2$, $S$ and $T$ be $n$-H-closed extensions of an $n$-Hausdorff space $X$. We say $S$ is \textit{projectively larger} than $T$ if there is a continuous surjection $f : S \rightarrow T$ such that $f(x) = x$ for $x \in X$.
\end{definition}

This projectively larger function may not be unique \cite{BBC}.

\begin{theorem} \rm \cite{BBC}
\label{T5}
Let $n \geq 2$, $X$ be $n$-Hausdorff space and $Y$ be the $n$-H-closed extension of $X$ constructed above. If $Z$ is an $n$-H-closed extension of $X$, there is a continuous surjection $f : Y \rightarrow Z$ such that $f(x) = x$ for all $x \in X$.
\end{theorem}

Theorem \ref{T5} shows that the $n$-H-closed extension Y of X is projectively larger than every
$n$-H-closed extension of X. Moreover, the space Y has an interesting unique property as it is noted in the next result.

\begin{theorem} \rm \cite{BBC}
\label{T6}
Let $n \geq 2$, $X$ be an $n$-Hausdorff space and $Y$ be the $n$-H-closed extension of $X$ described above. Let $f : Y \rightarrow Y$ be a continuous surjection such that $f(x) = x$ for all $x \in X$. Then $f$ is a homeomorphism.
\end{theorem}

\begin{remark} \rm

In the class of Hausdorff spaces the function in Definition \ref{projlarg} is unique \cite{BBC}. Sometimes this is a problem in non-Hausdorff spaces. The $n$-H-closed space $Y$ constructed before for an $n$-Hausdorff space $X$ is a projective maximum, that is $Y$ is projectively larger than every $n$-H-closed extention and given a continuous surjection $f: Y\to Y$ such that $f(x)=x$ for every $x\in X$, then $f$ is a homeomorphism.  For the future we denote this $Y$ with $n\hbox{-}kX$ and we call it the \emph{n-Kat\v{e}tov extension} of $X$.
\end{remark}

Uspenski\v{i} showed in \cite{U} that for any space $X$ there exists a cardinal $\kappa$ and a nonempty subspace $Z \subseteq X^{\kappa}$ such that $X \times Z$ is homogeneous. The space $Z$ is found by selecting a set $A$ such that $\kappa=|A| \geq |X|$ and letting $Z = \{f \in \hbox{}^AX : \hbox{for each } x \in X, |f^{\leftarrow}(x)| = \kappa \}$, where $^AX$ is the space of all functions from $A$ to $X$. Both $Z$ and $X \times Z$ are homogeneous and homeomorphic. For our construction we write ${\bf H}(X) = X \times Z$ and consider $X$ as a subspace of ${\bf H}(X)$ \cite{CPR2}.

\begin{lemma} \rm \cite{CPR2}
\label{lemma1}
Let $X$ be a space and $h : X \rightarrow X$ be a homeomorphism and let $id_{Z}$ be the identity function on $Z$ . Then the function $h \times id_{Z} : {\bf H}(X) \rightarrow {\bf H}(X)$ is also a homeomorphism that extends $h$.
\end{lemma}

\begin{lemma} \rm
\label{lemma2}
Let $n \geq 2$, $X$ be an $n$-Hausdorff space and $h:X\to X$ be a homeomorphism. Then there is a homeomorphism $n$-$kh:n\hbox{-}kX\to n\hbox{-}kX$ that extends $h$.
\end{lemma}
\begin{proof}
	Let $p\in n\hbox{-}kX\setminus X$, then $p=p_{\alpha i}$ for some $\alpha\in |\frak U|$ and for some $i=1,...,k\alpha$. The set ${\cal V} =\{h(U): U \in \mathcal{U}_\alpha\}$ is an open ultrafilter on $X$ and since $|a {\cal U}_\alpha|=|a\mathcal{V}|$, there exists $\beta \in |\frak U|$ such that $\mathcal{V}=\mathcal{U}_\beta$. Define $n\hbox{-}kh(p_{\alpha i})=p_{\beta i}$ for every $i=1,..., k\alpha=k\beta$. For $x\in X$, define $n\hbox{-}kh(x)=h(x)$. The function $n\hbox{-}kh$ is clearly a homeomorphism that extends $h$.
\end{proof}

\begin{theorem}\rm \label{Theorem}
	Let $n \geq 2$, $X$ be an $n$-Hausdorff space. Then $X$ can be embedded in an homogeneous space that is the countable union of $n$-H-closed spaces.
\end{theorem}
\begin{proof}
	Let $H_1={\bf H}(n\hbox{-}kX)$. If $H_{m}$ is defined, let's define $H_{m+1}={\bf H}(n\hbox{-}kH_m)$ and $H=\bigcup_{m\in \Bbb N}H_m$. A subset $U\subseteq H$ is open in $H$ if $U\cap H_m\in \tau(H_m)$ for every $m\in \Bbb N$. The space $H$ is the countable union of $n$-H-closed spaces. We have to prove that $H$ is homogeneous. Let $p,q\in H$. Since $H_m\subseteq H_{m+1}$, there exists $m\in \Bbb N$ such that $p,q\in H_m$. Each $H_m$ is homogeneous, then there exist a homeomorphism $h:H_m\to H_m$ such that $h(p)=q$. By Lemma \ref{lemma2} there exists a homeomorphism $n\hbox{-}kh:n\hbox{-}kH_m\to n\hbox{-}kH_m$ that extends $h$. By Lemma \ref{lemma1} the function $n\hbox{-}kh\times id_Z: H_{m+1}\to H_{m+1}$ is a homeomorphism. Put $n\hbox{-}kh=h_1$. By induction $h$ can be extended to $h_k:H_{m+k}\to H_{m+k}$ for every $k\in {\Bbb N}$. The function $g=\bigcup_{k\in {\Bbb N}}h_k: H\to H$ extends $h$ and it is a homeomorphism on $H$. Then $H$ is homogeneous.
\end{proof}

\begin{example}\rm
	An example of an $n$-Hausdorff, homogeneous, not $n$-Urysohn space which is the countable union of $n$-H-closed spaces, for every $n\geq 2$.
\end{example}
	Let's take an $n$-Hausdorff, not $n$-Urysohn space $X$ (for example see \cite[Example 4]{B}), $n\geq 2$. Then, by Theorem \ref{Theorem}, $X$ can be embedded in an $n$-Hausdorff, homogeneous space $Y$ which is the countable union of $n$-H-closed spaces. Furthermore $Y$ is not $n$-Urysohn, since $X$ is a non-$n$-Urysohn subset of it.\hfill$\triangle$\\
\smallskip

{\bf Acknowledgement:} The research was supported by "National Group for Algebric and Geometric Structures, and their Applications" (GNSAGA-INdAM).

\end{document}